\documentclass[a4paper,10pt]{amsart}

\input xy
\xyoption{all}

\usepackage{xcolor}
\usepackage{comment}
\usepackage{geometry}
\usepackage{amsmath}
\usepackage{amssymb}
\usepackage{cite}
\usepackage{amsthm}  
\usepackage{mathtools}
\usepackage{amsmath, amsthm, amssymb}
\usepackage{listings}
\usepackage{url}
\usepackage{hyperref}
\usepackage[utf8]{inputenc}
\usepackage{graphicx}
\usepackage{amsfonts}
\usepackage{tikz-cd}

\newcommand{\K}{\mathbf{K}}
\newcommand{\LS}{\operatorname{LS}}
\newcommand{\lea}{\leap{\K}}
\newcommand{\Ii}{\mathbb{I}}

\theoremstyle{definition}
\newtheorem{mydef}{Definition}[section]
\newtheorem{lem}[mydef]{Lemma}
\newtheorem{thm}[mydef]{Theorem}
\newtheorem{cor}[mydef]{Corollary}

\newtheorem{question}[mydef]{Question}
\newtheorem{hypothesis}[mydef]{Hypothesis}

\newtheorem{defin}[mydef]{Definition}

\newtheorem{remark}[mydef]{Remark}

\newtheorem{fact}[mydef]{Fact}
\newtheorem{prob}[mydef]{Problem}

\newtheorem*{thm1}{Theorem \ref{++}}
\newtheorem*{thm2}{Theorem \ref{main-2}}
\newtheorem*{thm3}{Theorem \ref{*}}
\newtheorem*{thm4}{Theorem \ref{trans}}
\newtheorem*{sec-1}{Shelah's categoricity conjecture}

\newcommand{\leap}[1]{\le_{#1}}

\newcommand{\fct}[2]{{}^{#1}#2}

\newcommand{\gS}{\mathbf{S}}

\newcommand{\rest}{\upharpoonright}

\newcommand{\concat}{%
  \mathord{
    \mathchoice
    {\raisebox{1ex}{\scalebox{.7}{$\frown$}}}
    {\raisebox{1ex}{\scalebox{.7}{$\frown$}}}
    {\raisebox{.5ex}{\scalebox{.5}{$\frown$}}}
    {\raisebox{.5ex}{\scalebox{.5}{$\frown$}}}
  }
}

\newcommand{\s}{\mathfrak{s}}
\newbox\noforkbox \newdimen\forklinewidth
\forklinewidth=0.3pt \setbox0\hbox{$\textstyle\smile$}
\setbox1\hbox to \wd0{\hfil\vrule width \forklinewidth depth-2pt
 height 10pt \hfil}
\wd1=0 cm \setbox\noforkbox\hbox{\lower 2pt\box1\lower
2pt\box0\relax}
\def\unionstick{\mathop{\copy\noforkbox}\limits}
\def\nf{\unionstick}

\setbox0\hbox{$\textstyle\smile$}
\setbox1\hbox to \wd0{\hfil{\sl /\/}\hfil} \setbox2\hbox to
\wd0{\hfil\vrule height 10pt depth -2pt width
               \forklinewidth\hfil}
\newbox\doesforkbox
\setbox\doesforkbox\hbox{\lower 0pt\box1 \lower
2pt\box2\lower2pt\box0\relax}

\def\1nf{\unionstick^{(1)}}

\def\2nf{\unionstick^{(2)}}

\def\3nf{\unionstick^{(3)}}

\newcommand{\dnf}{\unionstick}

\def\forkindep{\mathrel{\raise0.2ex\hbox{\ooalign{\hidewidth$\vert$\hidewidth\cr\raise-0.9ex\hbox{$\smile$}}}}}

\newcommand{\indep}[4]{#2 \overset{#4}{\underset{#1}{\forkindep}}  #3}
\newcommand{\type}{\mathbf{gtp}}
\newcommand{\gtp}{\mathbf{gtp}}
\title{On Stability and Existence of Models in Abstract Elementary Classes}

\date{\today} 

\author{Marcos Mazari-Armida}
\email{marcos\_mazari@baylor.edu}
\urladdr{https://sites.baylor.edu/marcos\_mazari/}
\address{Department of Mathematics \\ Baylor University \\ Waco, Texas, USA}
\thanks{The first author's research  was partially supported by an NSF grant DMS-2348881, a grant from the Simons Foundation (Travel Support for Mathematicians) and an AMS-Simons Travel Grant 2022--2024.}

\author{Sebastien Vasey}
\email{vasey.sebastien@gmail.com}
\urladdr{https://svasey.com}

\author{Wentao Yang}
\email{ndwyang@imu.edu.cn}
\urladdr{https://wen-tao-y.github.io}
\address{School of Mathematical Sciences \\ Inner Mongolia University \\ Hohhot, Inner Mongolia, China}

\begin{document}

\begin{abstract} 
For an abstract elementary class $\K$ and a cardinal $\lambda \geq \LS(\K)$, we prove under mild cardinal arithmetic assumptions, categoricity in two succesive cardinals, almost stability for $\lambda^+$-minimal types and continuity of splitting in $\lambda$,  that stability in $\lambda$ is equivalent to the existence of a model in $\lambda^{++}$. The forward direction holds without any cardinal or categoricity  assumptions, this result improves both \cite[12.1]{vaseyt} and \cite[3.14]{may}. 

Moreover, we prove a categoricity theorem for abstract elementary classes with weak amalgamation and tameness under mild structural assumptions in $\lambda$. A key feature of this result is that we do not assume amalgamation or arbitrarily large models. 


\end{abstract}

\maketitle
 
\tableofcontents
\section{Introduction}
Abstract elementary classes (AECs for short) were introduced by Shelah \cite{sh87a} as a model theoretic framework to study non-elementary classes. The two most important test problems in Shelah's program of \emph{classification theory for abstract elementary classes} \cite{shelahaecbook, grossberg2002} are: Grossberg's question on existence \cite[Problem (5), p. 34]{sh576} and Shelah's categoricity conjecture \cite{sh87a}, \cite[6.13.(3)]{sh704}.

Grossberg's question on existence, which generalizes Baldwin's question \cite{friedman}\footnote{Solved by Shelah for $PC_{\aleph_0}$ classes in\cite{Sh:88}.} to AECs, is the following:

\begin{prob}
Let $\mathbf K$ be an AEC and $\lambda\geq LS(\mathbf K)$ be an infinite cardinal. If $\mathbf K$ is categorical in $\lambda$ and $\lambda^+$, must $\mathbf K$ have a model  of cardinality $\lambda^{++}$?
\end{prob}

There are many partial solutions to this problem. An updated account of the state of Grossberg's question on existence can be consulted in the introduction of \cite{may}. Recently, there are many results where stability or superstability-like independence notions, with fewer or no categoricity assumptions, have been used to show the existence of larger models  \cite[\S II.4.13.3]{shelahaecbook}, \cite[3.1.9]{JaSh875},  \cite[8.9]{vasey-independence},  \cite[12.1]{vaseyt}, \cite[4.2]{m1}, \cite[3.11, 3.14]{may}. In this paper, we show that stability is enough to construct a model of cardinality $\lambda^{++}$ assuming no maximal models in $\lambda$, amalgamation in $\lambda$ and continuity of splitting in $\lambda$. \emph{Continuity of splitting}\footnote{Continuity of splitting has also been called \emph{continuity of non-splitting} in for example\cite{bv} and \cite{leu2}. In this paper, we use \emph{continuity of splitting}   following  \cite{vaseyt}.} (see Definition \ref{d-con}) is a mild assumption on  Galois types which follows from a weak locality condition on Galois types (see Lemma \ref{loc->con} and Remark \ref{exam}) or from superstability (see Lemma \ref{good-split} and Remark \ref{r-g-s}). In particular, continuity of splitting holds for $(<\aleph_0)$-tame AECs which extend elementary classes.

\begin{thm1}
Let $\K$ be an AEC and let $\lambda \ge \LS(\K)$. If  $\K$ has amalgamation in $\lambda$,  no maximal model in $\lambda$ and is stable in $\lambda$, and  splitting is continuous in $\lambda$, then $\K$ has a model in $\lambda^{++}$. 

\end{thm1}

Theorem \ref{++} improves both  \cite[12.1]{vaseyt} and \cite[3.14]{may}. Theorem \ref{++} improves  Vasey's result \cite[12.1]{vaseyt}, regarding existence of a model in $\lambda^{++}$, as we replace the categoricity and tameness assumptions by the weaker assumption of continuity of splitting.\footnote{Observe that $\lambda$-superstability in the sense of \cite[4.23]{vaseyt} implies continuity of splitting in $\lambda$ (see Remark \ref{r-g-s}).} Theorem \ref{++} improves Mazari-Armida's and Yang's result \cite[3.14]{may} as we replace the locality assumption  on Galois types by the weaker assumption of continuity of splitting and we drop the assumption that $\lambda < 2^{\aleph_0}$.  Furthermore, Theorem \ref{++} significantly improves \cite[\S II.4.13.3]{shelahaecbook} for AECs with the additional assumption of continuity of splitting in $\lambda$. We replace the existence of a good $\lambda$-frame which is a superstability-like independence notion by the weaker assumption of stability in $\lambda$.\footnote{Observe that the existence of good $\lambda$-frame implies stability in $\lambda$ by \cite[\S II.4.2]{shelahaecbook}.} The additional assumption of continuity of splitting in $\lambda$ is a minor assumption in the context of \cite[\S II.4.13.3]{shelahaecbook} as it follows from the existence of a type-full good $\lambda$-frame (see Lemma \ref{good-split}).

Theorem \ref{++} reduces to showing that the class has no maximal model in $\lambda^+$. We achieve this by building  a coherent sequence of types which does not split over a base model. We use continuity of splitting  to carry out the construction at limit stages. The limit of the sequence is a non-algebraic type and hence the AEC has no maximal models in $\lambda^+$.

In this paper, we further show that stability is necessary to construct a model of cardinality $\lambda^{++}$ under mild cardinal arithmetic assumptions, categoricity in two succesive cardinals and almost stability for $\lambda^+$-minimal types. Observe that this result does not assume continuity of splitting.

\begin{thm2} Suppose that $\lambda^+ < 2^{\lambda} < 2^{\lambda^+}$. Let $\K$ be an AEC and let $\lambda \ge \LS(\K)$.  Assume $\K$ is categorical in $\lambda$ and $\lambda^+$, $\K_{\lambda^{++}} \neq \emptyset$ and $|\gS^{\neg\lambda^+-min}(M)|\leq \lambda^+$ for the unique model $M\in \K_\lambda$. Then $\K$ is stable in $\lambda$.
\end{thm2}

When $\K$ is an elementary elementary class, Morley, in an intermediate step of proving his categoricity theorem \cite{morley}, showed that categoricity in an uncountable cardinal implies stability in $\aleph_0$. There are similar results for AECs assuming arbitrarily large models, see for example \cite{sh394}, \cite{vasey-independence} and \cite{BGVV}.  Theorem \ref{main-2} is a local version of these results for abstract elementary classes. A key feature of our result is that we do not assume arbitrarily large models.

 There are not many results which obtain stability without the assumption of arbitrarily large models. One is \cite[1.1]{bls}, where the authors prove stability in $\aleph_0$ from categoricity in $\aleph_1$ and existence of a model of cardinality $(2^{\aleph_0})^+$ for atomic classes. Comparing our result to \cite{bls}, note that our result does not need $\lambda$ to be $\aleph_0$ and we work in the more general context of abstract elementary classes. We do not assume the existence of a model in $(2^{\aleph_0})^+$, but we have extra assumptions on the number of types and more importantly we have a cardinal arithmetic assumption beyond ZFC. 
 
 The proof of Theorem \ref{main-2} is done in two steps. First we use \cite[ VI.2.11]{shelahaecbook} to prove that the abstract elementary class is almost stable in $\lambda$ (see Theorem \ref{almost-stability}). Then, we show that there is a minimal type and hence the AEC is stable in $\lambda$ by \cite[VI.5.3.(1)]{shelahaecbook} (see Fact \ref{5.3} of this paper).
 

Combining Theorem \ref{++} and Theorem \ref{main-2}, one obtains the result mentioned in the abstract.

\begin{thm3} Suppose $\lambda^+ < 2^{\lambda} < 2^{\lambda^+}$. Let $\K$ be an AEC and let $\lambda \ge \LS(\K)$.
 Assume $\K$ is categorical in $\lambda$ and $\lambda^+$, $\K$ is almost stable for non-$\lambda^+$-minimal types in $\lambda$ and splitting is continuous in $\lambda$. The following are equivalent.
 
 \begin{enumerate}
 \item  $\K$ has a model in $\lambda^{++}$.
 \item $\K$ is stable in $\lambda$.
 \end{enumerate}
 \end{thm3}

 \emph{Shelah's categoricity conjecture}, was posed by  Shelah in \cite{sh87a}, \cite[6.13.(3)]{sh704}. An updated account of the state of the conjecture can be consulted in the introductions of \cite{vasey18} and \cite{mazcat}. In this paper, we obtain the following categoricity transfer result. 

\begin{thm4}
Let $\K$ be an AEC with weak amalgamation and let $\lambda \ge \LS(\K)$ be such that $\K$ is $\lambda$-tame. Assume  $\K$ has amalgamation in $\lambda$, $\K$ is stable in $\lambda$, and  splitting is continuous in $\lambda$. If $\K$ is categorical in $\lambda$ and $\lambda^+$, then $\K$ is categorical in all $\mu \ge \lambda$.
\end{thm4}

Theorem \ref{trans} can be seen as a generalization of the classical result of Grossberg and VanDieren \cite[4.3]{tamenesstwo}. We replace the assumption that the AEC has arbitrarily large models and amalgamation by the weaker assumption of weak amalgamation (see Definition \ref{w-ap}). The rest of our other assumptions,  stability in $\lambda$, amalgamation in $\lambda$, and continuity of splitting in $\lambda$, are easily derivable in the Grossberg-VanDieren context.  In \cite{sebastien-successive-categ}, it was already observed that the Grossberg-VanDieren result carries through if only weak amalgamation is assumed, but here we do not even assume arbitrarily large models. 



 Theorem \ref{trans} reduces to showing arbitrarily large models and amalgamation as then one can apply Grossberg's and VanDieren's results \cite[5.2, 6.3]{tamenessthree}. We show arbitrarily large models and amalgamation by  first building a w-good $\lambda^+$-frame with density  and then extending it to a w-good $[\lambda^+,\infty)$-frame with density (see Lemma \ref{m-lemma}). W-good frames were introduced in \cite{m1} as weakening of Shelah's key notion of a good frame  \cite[\S II.2, p. 259-263]{shelahaecbook}.

The paper is organized as follows. Section 2 presents necessary background. Section 3 presents a partial answer to Grossberg's question on existence assuming stability in $\lambda$ and continuity of splitting in $\lambda$. Moreover, Section 3 presents the aforementioned categoricity transfer result (Theorem \ref{trans}). Section 4 presents the argument that stability in $\lambda$ follows from the existence of a model of cardinality $\lambda^{++}$ under mild cardinal arithmetic assumptions, categoricity in two succesive cardinals and almost stability for $\lambda^+$-minimal types. It is worth emphasising that Section 3 and 4 are basically independent of each other. The reader only interested in  the results of Section 4 can go directly to Section 4.

This paper was partially written while the third author was working on a Ph.D. thesis under the direction of Rami Grossberg at Carnegie Mellon University, and the third author would like to thank Professor Grossberg for his guidance and assistance in his research in general and in this work
specifically. A version of this paper is Chapter 4 of Yang's PhD thesis \cite{wyt}. We would  like to thank Jeremy Beard for discussions that helped improve the paper. We would like to thank an anonymous referee for comments that help improve the presentation of the paper.

\section{Prelimianries}

\subsection{Basic notions}
We assume that the reader is already familiar with the basic definitions of abstract elementary classes, but we quickly review the basic notions used in this paper.  These are further studied in \cite[\S 4 - 8]{baldwinbook09} and  \cite[\S 2, \S 4.4]{ramibook}. 

An abstract elementary class is a pair $\K=(K, \lea)$ where $K$ is a class of structures in a fixed language and $\lea$ is a partial order on $K$ extending the substructure relation such that $\K$ is closed under isomorphisms and satisfies the  coherence property, the L\"{o}wenheim-Skolem-Tarski axiom and the Tarski-Vaught chain axioms. The L\"{o}wenheim-Skolem-Tarski axiom states that the class satisfies an instance of the Downward L\"{o}wenheim-Skolem theorem and the Tarski-Vaught chain axioms state  that the
class is closed under directed colimits. The reader can consult the definition in \cite[4.1]{baldwinbook09}. We denote the L\"owenheim-Skolem-Tarski number of $\K$ by $LS(\K)$.  

For a cardinal $\mu$, let $\K_\mu$ denote the set of structures in $\K$ with cardinality $\mu$. For a structure $M\in \K$ denote its universe by $|M|$ and its cardinality $\|M\|$. For an AEC $\K$ and  $\lambda \geq \LS(\K)$, we denote by $\Ii(\K, \lambda)$ the number of models in $\K_\lambda$ up to isomorphism. If  $\Ii(\K, \lambda)=1$, we say that $\K$ is categorical in $\lambda$. 

For an AEC $\K$, $\K$ has the amalgamation if for every $M_0\lea M_l$ for $\ell=1,2$, there is $N\in \K$ and $\K$-embeddings $f_\ell:M_\ell\to N$ for $\ell=1,2$ such that $f_1 \restriction {M_0}=f_2 \restriction {M_0}$. We say that $\K$ has amalgamation in $\lambda$, or that $\K_\lambda$ has amalgamation when the condition holds only for models of size $\lambda$.

\begin{hypothesis}
We fix an abstract elementary class $(\K,\lea)$ throughout the paper and $\lambda\geq LS(\K)$ an infinite cardinal. .
\end{hypothesis}

\begin{fact}[{\cite[3.5]{Sh:88}, \cite[4.3]{grossberg2002}}]\label{few_models->ap}
 Suppose $2^\lambda<2^{\lambda^+}$. If $\Ii(\K, \lambda)=1\leq \Ii(\K, \lambda^+)<2^{\lambda^+}$, then $\K$ has amalgamation in $\lambda$.
\end{fact}

For an AEC $\K$, $\K$ has no maximal models if every model has a proper extension and has joint embedding if any two models can be $\K$-embedded into a third model. 

\subsection{Galois types and related properties}
We also assume that the reader is familiar with Galois types, which are sometimes called orbital types in the literature. For $M<_\K N\in \K$ and $a\in |N|$, the Galois type of $a$ over $M$ is denoted by $\type(a/M,N)$. Let $\gS(M)$ denote the set of all Galois types over $M$. When $a\notin |M|$, we say that $\type(a/M,N)$ is non-algebraic and write $\type(a/M,N)\in \gS^{na}(M)$.

\begin{mydef}\
\begin{enumerate}
    \item (\cite{JaSh875})
    $\K$ is \emph{almost stable in $\lambda$} if $|\gS(M)| \le \lambda^+$ for all $M \in \K_\lambda$.
    \item $\K$ is \emph{stable in $\lambda$} if $|\gS(M)| \le \lambda$ for all $M \in \K_\lambda$.
\end{enumerate}
\end{mydef}

\begin{mydef}
$\gS_*$ is  \emph{$\leq_{\K_\lambda}$-type-kind} when:
\begin{enumerate}
    \item $\gS_*$ is a function with domain $\K_\lambda$.
    \item $\gS_*(M)\subseteq \gS^{na}(M)$ for every $M\in \K_\lambda$.
    \item $\gS_*(M)$ commutes with isomorphisms for every $M\in \K_\lambda$.
\end{enumerate}
\end{mydef}
\begin{mydef}[{\cite[VI.1.12(2)]{shelahaecbook}}] 
$\gS_1$ is \emph{hereditarily} when: for $M\lea N \in \K_\lambda$ and $p\in \gS^{na}(N)$ we have that if $p\restriction M\in \gS_1(M)$ then $p\in \gS_1(N)$. 
\end{mydef}

\begin{mydef}[{\cite[VI.2.12]{shelahaecbook}}]
    For $M\in \K_\lambda$ and $p\in \gS^{na}(M)$, we say $p$ is \emph{$\mu$-minimal} if for all $M\lea N\in \K_\lambda$, $|\{q\in \gS^{na}(N): q\restriction M=p\}|\leq \mu$. We denote the set of $\mu$-minimal types by $\gS^{\mu-min}(M)$ and those that are not by $\gS^{\neg \mu-min}(M)$. We say $p$ is minimal when $p$ is $1$-minimal.
\end{mydef}

\begin{remark}
    $\gS^{\mu-min}$ is a $\lea$-type-kind  and hereditary for any cardinal $\mu$.
\end{remark}

\begin{mydef}\
\begin{enumerate}

    \item $\K$ is \emph{$(< \lambda^+, \lambda)$-local} if for every $M \in \K_\lambda$, every $\kappa < \lambda^+$, every increasing continuous chain $\langle M_i : i<\kappa \rangle$ such that  $M=\bigcup_{i<\kappa}M_i$ and every $p,q\in \gS(M)$, if $p\restriction {M_i}=q\restriction {M_i}$ for all $i < \kappa$ then $p=q$.
        \item $\K$ is \emph{$\lambda$-tame} if for every $M \in \K$ and and every $p,q\in \gS(M)$, if $p\neq q$, then there is $N \lea M$ of cardinality  $\lambda$ such that $p\restriction_N \neq q\restriction_N$.
\end{enumerate}

\end{mydef}

 \begin{remark}\label{coh} We will often use that given a coherent sequence of types $\langle p_i\in \gS^{na}(M_i):  i<\delta\rangle$ there is $p\in \gS^{na}(M_\delta)$ such that $p \geq p_i$ for every $i < \delta$ and $\langle p_i\in \gS^{na}(M_i) : i < \delta + 1\rangle$ is coherent. See for example \cite[3.14]{m1}.
 \end{remark}

\begin{mydef}
    Let $\kappa$ be a cardinal and $M\in \K_\lambda$. Let $\gS_*$ be a $\lea$-type kind. Suppose $\Gamma \subseteq \gS^{na}(M)$. We say $\Gamma$ is \emph{$\gS_*$-inevitable} if for all $M\lea N\in\K_\lambda$, if there is $a\in |N|-|M|$ with $\type(a/M,N)\in \gS_*(M)$, then there is $b\in | N | -|M|$ such that $\type(b/M,N)\in \Gamma$.
\end{mydef}

\begin{mydef}
Let $\gS_*$ be  \emph{$\leq_{\K_\lambda}$-type-kind}.  $M$ saturated for $\gS_*$-types in $\lambda^+$ above $\lambda$ if $M \in \K_{\lambda^+}$ and for every  $M_0<_\K M$ with $M_0 \in \K_\lambda$ and $p \in \gS_*(M_0)$, $p$ is realized in $M$. $M$ is saturated  in $\lambda^+$ above $\lambda$ if it is saturated for $\gS^{na}$-types in $\lambda^+$ above $\lambda$.
\end{mydef}

 $M$ is  $\lambda^+$-model-homogeneous above $\lambda$ if for every  $M_0<_\K M$ with $M_0 \in \K_\lambda$  and $M_0 \lea M_1 \in \K_\lambda$, there is $f: M_1 \xrightarrow[M_0]{} M$ a $\K$-embedding, i.e., $f\restriction M_0 = id_{M_0}$. For $\lambda > \LS(K)$,  $M$ is  saturated  in $\lambda^+$ above $\lambda$ if and only if $M$ is $\lambda^+$-model-homogeneous above $\lambda$. See for example \cite[\S II.1.4]{shelahaecbook}.

We finish by introducing universal models.

\begin{mydef}
$M$ is \emph{universal over} $N$ if and only if $\|M\| = \|N\| = \mu$, $N \leq_\K M$
 and for any $N^* \in \K_{\mu}$ such that
$N \leq_\K N^*$, there is $f: N^* \xrightarrow[N]{} M$ a $\K$-embedding
\end{mydef}

We will often use in Section 3 that if $\K$ is an AEC  with amalgamation, no maximal models and joint embedding and $\K$ is stable in $\lambda$, then for any $M \in \K_\lambda$, there is $N$ universal over $M$.   See \cite[\S II]{shelahaecbook}, \cite[2.9]{tamenessone}.

\subsection{Good frames}  Good frames were introduced by Shelah in \cite[\S II.2, p. 259-263]{shelahaecbook}. We will work with  w-good frames in this paper. This a a weakening of the notion of a good frame which was introduced in \cite{m1}. W-good frames are only used in Section 3 of the paper.

\begin{mydef}\label{preframe} 
Let $\lambda<\mu$, where $\lambda$ is a cardinal, and $\mu$ is a cardinal or $\infty$. A \emph{w-good $[\lambda,\mu)$-frame} is a triple $\mathfrak s=(\K,\nf{},\gS^{bs})$ such that:
\begin{enumerate}
    \item $\K$ is an AEC with $\lambda\geq LS(\K)$ and $\K_\lambda\neq \emptyset$.
    \item $\K_{[\lambda,\mu)}$ has amalgamation, the joint embedding property and no maximal models.
    \item $\gS^{bs}\subseteq \bigcup_{M\in \K_{[\lambda,\mu)}}\gS^{na}(M)$. Let $\gS^{bs}(M):=\gS(M)\cap \gS^{bs}$. Types in this family are called \emph{basic types}.
    \item $\nf$ is a relation on quadruples $(M_0,M_1,a,N)$, where $M_0\lea M_1\lea N$, $a\in |N|$ and $M_0,M_1,N\in \K_{[\lambda,\mu)}$. We write $\indep{M_0}{a}{M_1}{N}$, or we say $\type(a/M_1,N)$ does not fork over $M_0$ when the relation $\nf$ holds for $(M_0,M_1,a,N)$.
    \item (Invariance) If $f:N\cong N'$ and $\indep{M_0}{a}{M_1}{N}$, then $\indep{f[M_0]}{f(a)}{f[M_1]}{N'}$. If $\type(a/M_1,N)\in \gS^{bs}(M_1)$, then $\type(f(a)/f[M_1],N')\in \gS^{bs}(f[M_1])$.
    \item (Monotonicity) If $\indep{M_0}{a}{M_1}{N}$ and $M_0\lea M_0'\leq_K M_1'\lea M_1\lea N'\lea N\lea N''$ with $N''\in \K_{[\lambda,\mu)}$ and $a\in |N'|$, then $\indep{M_0'}{a}{M_1'}{N'}$ and $\indep{M_0'}{a}{M_1'}{N''}$.
    \item (Non-forking Types are Basic) If $\indep{M}{a}{M}{N}$ then $\type(a/M,N)\in \gS^{bs}(M).$
    \item (Weak Density) For all $M<_\K N\in \K_\lambda$, there is $a\in |N|-|M|$ and $M'\lea N'\in \K_{[\lambda,\mu)}$ such that $(a,M,N)\leq (a,M',N')$ and $\type(a/M',N')\in \gS^{bs}(M')$.
    \item (Existence of Non-Forking Extension) If $p\in \gS^{bs}(M)$ and $M\leq_K N$, then there is $q\in \gS^{bs}(N)$ extending $p$ which does not fork over $M$.
    \item (Uniqueness) If $M\lea N$ both in $\K_{[\lambda,\mu)}$, $p,q\in \gS^{bs}(N)$ both do not fork over $M$, and $p\restriction M=q\restriction M$, then $p=q$.
    \item (Continuity) If $\delta<\mu$ a limit ordinal, $\langle M_i\mid i\leq \delta\rangle$ increasing and continuous, $\langle p_i\in \gS^{bs}(M_i)\mid i<\delta \rangle$, and $i<j<\delta$ implies $p_j\restriction M_i=p_i$, and $p_\delta\in \gS(M_\delta)$ is an upper bound for $\langle p_i\mid i<\delta\rangle$, then $p\in \gS^{bs}(M_\delta)$. Moreover, if each $p_i$ does not fork over $M_0$ then neither does $p_\delta$.
\end{enumerate}
\end{mydef}

We say that a  w-good $[\lambda,\mu)$-frame has \emph{density} if for every $M <_\K N$  both in $\K_{[\lambda, \mu)}$, there is an $a \in |N|$ such that $\gtp(a/M , N) \in \gS^{bs}(M)$.
 \begin{remark}
     Note that density and inevitability, both introduced in \cite[\S III, \S VI]{shelahaecbook}, are the same notion when $\gS_*=\gS^{bs}$. The former is usually used in the context of frames.
 \end{remark}

\subsection{Splitting} We introduce the basic properties of splitting we will use in this paper. Recall that splitting for AECs was introduced in \cite[Definition 3.2]{sh394}. Splitting is only used in Section 3 of the paper.

\begin{mydef} Let $M \in \K_\lambda$, $M \lea N$ and  $p \in \gS(N)$ . $p$ ($\lambda$-)splits over $M$ if there are $N_1, N_2 \in \K_\lambda$ and $h: N_1 \cong_ M N_2$ such that $M \lea N_1, N_2 \lea N$ and $h(p\rest N_1) \neq p\rest N_2$. 
\end{mydef} 

We will use the following properties of non-splitting often in this section.
\begin{fact}
Assume  $\K$ has amalgamation in $\lambda$ and no maximal model in $\lambda$. 
\begin{enumerate} 
\item (\cite[3.3]{sebastien-s-first-good-frame-construction-paper}) Monotonicity: If $M_0 \lea M_1 \lea M_2 \lea M_3$, $p \in \gS(M_3)$ does not split over $M_0$ and $M_0, M_1, M_2 \in \K_\lambda$, then $p\rest M_2$ does not split over $M_1$. 
\item Let $M_0 \lea M_1 \lea M_2$ all in $\K_\lambda$ and $M_1$ is universal over $M_0$.
\begin{itemize}
\item  (\cite[I.4.10]{van06})  Weak extension: If $p \in \gS^{na}(M_1)$ does not split over $M_0$, then there is $q \in \gS^{na}(M_2)$ such that $q$ extends $p$ and $q$ does not split over $M_0$. 
\item (\cite[I.4.12]{van06}) Weak uniqueness: If  $p, q \in \gS(M_2)$,  $p\rest M_1 = q\rest M_1$, and  $p, q$ do not split over $M_0$, then $p =q$.
\end{itemize}

\item (\cite[3.7]{sebastien-s-first-good-frame-construction-paper}) Weak transitivity: If $M_0 \lea M_1 \lea M_1' \lea M_2$ all in $\K_\lambda$,  $M_1'$ universal over $M_1$ and $p \in \gS^{na}(M_2)$ such that $p$ does not split over $M_1$ and $p\rest M_1'$ does not split over $M_0$, then $p$ does not split over $M_0$. 
\end{enumerate}
\end{fact}

\begin{fact}[{\cite[3.3]{sh394}, \cite[Theorem 2.2.1]{shvi} , \cite[3.8]{leu2}
}]\label{wlocal} (Weak universal local character)
Assume  $\K$ has amalgamation,  no maximal model and is stable in $\lambda$.  If $ \langle M_i : i \leq \lambda \rangle$  is an increasing continuous chain in $\K_\lambda$  with $M_{i + 1}$  universal over $M_i$ for all $i < \lambda$ and $p \in \gS^{na}(M_\lambda)$, then there is  $i < \lambda$ such that $p \rest M_{i + 1}$ does not split over $M_i$.
\end{fact}

\subsection{The weak diamond}

The following principle known as the weak diamond was introduced by Devlin and Shelah \cite{sh65}. The weak diamond is only used in Section 4 of the paper.

\begin{mydef}

Let $S\subseteq \lambda^+$ be a stationary set.  $\Phi_{\lambda^+}^{k}(S)$ holds if and only if for all $ F: {}^{<\lambda^+}(2^{\lambda})  \rightarrow k$  there exists $ g: \lambda^+ \rightarrow k$ such that for all $ f: \lambda^+ \rightarrow 2^{\lambda}$ the set $\{ \alpha \in S : F(f \restriction {\alpha}) = g(\alpha) \}$ is stationary. When $S=\lambda^+$ we write $\Phi^k_{\lambda^+}$ for $\Phi^k_{\lambda^+}(S)$.
\end{mydef}

The proofs of the following two facts can be consulted in \cite[\S 15]{ramibook}.
\begin{fact}
    Let $S\subseteq \lambda^+$ be a stationary set and $k<\omega$. If $\Phi^k_{\lambda^+}(S)$ holds, then for all $$F:\underbrace{ {}^{<\lambda^+} (2^\lambda)\times \ldots \times {}^{<\lambda^+} (2^\lambda)}_{n\text{ times}}\to k$$ there is $ g: \lambda^+ \rightarrow k$ such that for all $ f_i: \lambda^+ \rightarrow 2^{\lambda}$ for $i<n$ the set $$\{ \alpha \in S : F(f_1 \restriction {\alpha}, \ldots, f_{n-1} \restriction \alpha) = g(\alpha) \}$$ is stationary. 
\end{fact}

\begin{fact}\label{Ulam}
\
\begin{enumerate}
\item $2^{\lambda} < 2^{\lambda^+}$ if and only if $\Phi_{\lambda^+}^{2}(\lambda^+)$ holds.

\item Suppose that $\Phi^2_{\lambda^+}$ holds. Then there are disjoint stationary sets $S_\alpha \subseteq \lambda^+$ for $\alpha<\lambda^+$ such that $\Phi^2_{\lambda^+}(S_\alpha)$ holds for all $\alpha<\lambda^+$.
\end{enumerate}
\end{fact}

\section{Existence and categoricity above $\lambda^{++}$}

We provide a partial answer to Grossberg's question on existence and obtain a categoricity transfer result for AECs where splitting is continuous in $\lambda$.

\subsection{Existence of a model in $\lambda^{++}$} We focus first on Grossberg's question on existence in $\lambda^{++}$.

\begin{mydef}\label{d-con}
Assume $\K$ is stable in $\lambda$. Splitting is \emph{continuous in $\lambda$} if for any limit ordinal $\delta < \lambda^{+}$ and any increasing continuous chain $\langle M_i : i \le \delta\rangle$ with $M_{i + 1}$ universal over $M_i$ for all $i < \delta$, if $p \in \gS^{na}(M_\delta)$ is such that $p \rest M_i$ does not split over $M_0$ for all $i < \delta$, then $p$ does not split over $M_0$.
\end{mydef}

The following result is folklore, but we provide a proof as we could not find a reference. 

\begin{lem}\label{loc->con} Assume  $\K$ has amalgamation in $\lambda$, no maximal model in $\lambda$ and is stable in $\lambda$.
If $\K$ is $(<\lambda^+, \lambda)$-local, then splitting is continuous in $\lambda$.
\end{lem}
\begin{proof}
Let $\delta < \lambda^{+}$ be a limit ordinal,  $\langle M_i : i \le \delta\rangle$ be an increasing continuous chain with $M_{i + 1}$ universal over $M_i$ for all $i < \delta$ and $p \in \gS^{na}(M_\delta)$ such that $p \rest M_i$ does not split over $M_0$ for all $i < \delta$. 

Applying the weak extension property to $p\rest M_1$, there is $q \in \gS^{na}(M_\delta)$ such that $q$ extends $p\rest M_1$ and $q$ does not split over $M_0$. We show that $p =q$. Since $\K$ is $(<\lambda^+, \lambda)$-local, it is enough to show that $p\rest M_i=q\rest M_i$ for all $i < \delta$.

 Let $i< \delta$. When $i=0$ or $i=1$ the result is clear as  $p\rest M_1 \leq q$. When $i > 1$ the result follows from weak uniqueness and the fact that $q \rest M_i$ does not split over $M_0$ by monotonicity. 
\end{proof}

\begin{remark}\label{exam} Universal classes \cite[3.7]{universal-class-part-i},  Quasiminimal AECs (in the sense of \cite{vaseyq}) \cite[4.18]{vaseyq},   and many natural AECs of modules (see for example \cite[\S 3]{maz2}) are $(<\lambda^+, \lambda)$-local. Hence splitting is continuous in all of these classes by Lemma \ref{loc->con}. 
\end{remark}

We provide another condition which also implies continuity of splitting in $\lambda$.

\begin{lem}\label{good-split}
If $\K$ has a type-full good $\lambda$-frame, then  splitting is continuous in $\lambda$.
\end{lem}
\begin{proof}
Let $\delta < \lambda^{+}$ be a limit ordinal,  $\langle M_i : i \le \delta\rangle$ be an increasing continuous chain in $\K_\lambda$ with $M_{i + 1}$ universal over $M_i$ for all $i < \delta$ and $p \in \gS^{na}(M_\delta)$ such that $p \rest M_i$ does not split over $M_0$ for all $i < \delta$.

Then there is an $i <\delta$ such that $p$ does not fork over $M_i$ (in the sense of the good $\lambda$-frame) by local character of the good $\lambda$-frame. Hence $p$ does not split over $M_i$ by \cite[4.2]{bgkv} (see also \cite[4.16]{leu2}). Then it follows that $p$ does not split over $M_0$ by weak transitivity of splitting as $p$ does not split over $M_i$, $p \rest M_{i+1}$ does not split over $M_0$ and $M_{i+1}$ is universal over $M_i$.\end{proof}

\begin{remark}\label{r-g-s} Continuity of splitting also follows from $\lambda$-superstability in the sense of \cite[4.23]{vaseyt} by a similar argument to that of Lemma \ref{good-split} (see also \cite[3.16]{leu2}).  $\lambda$-superstability can be derived from categoricity, amalgamation, and arbitrarily large models \cite{BGVV} or from stability, categoricity, amalgamation and tameness \cite[12.1]{vaseyt}.
\end{remark}

\begin{thm}\label{++}
Let $\K$ be an AEC and let $\lambda \ge \LS(\K)$. If  $\K$ has amalgamation in $\lambda$, no maximal model in $\lambda$ and is stable in $\lambda$, and  splitting is continuous in $\lambda$, then $\K$ has a model in $\lambda^{++}$.

\end{thm}
\begin{proof}
We show $\K$ has no maximal models in $\lambda^+$. Assume for the sake of contradiction that there is $N \in \K_{\lambda^+}$ a maximal model. 


 First build a strictly increasing continuous chain $ \langle M_i : i \leq \lambda \rangle$ in $\K_\lambda$ with $M_{i + 1}$  universal over $M_i$ for all $i < \lambda$ and $M_i \lea N$ for every $i < \lambda$. This is possible by stability and amalgamation in $\lambda$ and the maximality of $N$.  Pick  $p \in \gS^{na}(M_\lambda)$. It follows from Fact \ref{wlocal}, that  there exists $i < \lambda$ such that $p \rest M_{i + 1}$ does not split over $M_i$.

Let $\{ n_i : i <\lambda^+ \}$ be an enumeration of $N$ and $N_* = M_i$. We build an increasing continuous chain $ \langle N_i : i < \lambda^+ \rangle$ in $\K_\lambda$ and $ \langle p_i :  i < \lambda^+ \rangle$ a chain of types such that:

\begin{enumerate}
\item  $N_0 = M_{i+1}$ and $p_0 =  p \rest M_{i + 1}$;
\item for every $i <\lambda^+$, $n_i \in N_{i+1}$, $N_i \lea N$ and $N_{i+1}$ is universal over $N_i$;
    \item for every $i < \lambda^+$,  $p_{i} \in \gS^{na}(N_{i})$ does not split over $N_*$;
 \item if $i< j< \lambda^+$, then $p_i\leq p_j$;
\item for every $j < \lambda^+$, $\langle p_i :  i<j\rangle$ is coherent . 
\end{enumerate}

\fbox{Construction} The base step is given by condition (1) and for $i$ limit, the construction can be carried out by coherence of the sequence and the fact that splitting is continuous in $\lambda$ by assumption. So we  do the case when $i=j+1$. Let $L$ be the structure obtained by applying the Löwenheim-Skolem-Tarski axiom to $N_{j} \cup  \{ n_j \}$ in $N$ and $L^* \in \K_\lambda$ a universal model over $L$, $L^*$  exists by stability and amalgamation in $\lambda$.  Using amalgamation in $\lambda$ and the maximality of $N$ there is $f: L^* \xrightarrow[L]{} N$. Let $N_{j+1} = f[L^*]$. As $N_{j}$ is universal over $N_*$, applying the weak extension property to $p_j$ one obtains $p_{j+1} \in \gS^{na}(N_{j+1})$ extending $p_j$ and such that $p_{j+1}$ does not split over $N_*$. It is easy to check that $N_{j+1}$ and $p_{j+1}$ satisfy conditions (2) to (5).

\fbox{Enough} Let  $p^*\in \gS^{na}(\bigcup_{i < \lambda^+} N_i)$ be an upper bound of the coherent sequence $\langle p_i :  i <\lambda^+\rangle$. Since $p^*$ is not algebraic by Remark \ref{coh} and $N = \bigcup_{i < \lambda^+} N_i$ by condition (2) of the construction, it follows that $N$ has a proper extension. This contradicts the assumption that $N$ was a maximal model. \end{proof}

\subsection{Categoricity above $\lambda^{++}$} We obtain a partial solution to the categoricity problem. A key assumption we will use to transfer categoricity that we did not have in the previous section is tameness. 

The following two results are known, but we could not find a reference so we sketch the proofs for the convenience of the reader.  

\begin{fact}\label{model-h}
If $\K$ has amalgamation in $\lambda$, $\K$ is stable in $\lambda$,  and $\K$ is categorical in $\lambda^+$, then the model of cardinality $\lambda^+$ is $\lambda^+$-model-homogeneous above $\lambda$. 
\end{fact}
\begin{proof}
We can assume without loss of generality that $\K$ has joint embedding and no maximal models in $\lambda$. If not, partition $\K_\lambda$ into equivalence classes given by $M$ is equivalent to  $N$ if they can be $\K$-embedded into a model in $\K_\lambda$, and restrict yourself to the class that generates the model in $\lambda^+$.

Build a strictly increasing continuous chain $ \langle M_i : i < \lambda^+ \rangle$ in $\K_\lambda$ with $M_{i + 1}$  universal over $M_i$ for all $i < \lambda$. This is possible by stability, joint embedding, no maximal,  and amalgamation in $\lambda$. Let $M_{\lambda^+} = \bigcup_{i<\lambda^+} M_i \in \K_{\lambda^+}$. Using a cofinality argument, it is clear that  $M_{\lambda^+}$ is model-homogeneous above $\lambda$. Therefore, the model of cardinality $\lambda^+$ is $\lambda^+$-model-homogeneous above $\lambda$. 
\end{proof}

\begin{fact}\label{w2}
Assume  $\K$ has amalgamation in $\lambda$,  no maximal model in $\lambda$ and is stable in $\lambda$ and $\K$ is categorical in $\lambda^+$. If $N \in \K_{\lambda^+}$ and $p \in \gS(N)$, then there is $M \in \K_\lambda$ such that $M \lea N$ and $p$ does not split over $M$. 
\end{fact}
\begin{proof}
    It follows from Fact \ref{wlocal} using that $N$ is $\lambda^+$-homogeneous above $\lambda$ by Fact \ref{model-h} and an analogous argument to that of \cite[12.5]{baldwinbook09}.
\end{proof}

The following weakening of amalgamation was isolated in \cite[4.11]{universal-class-part-i} and developed in \cite[\S 4]{universal-class-part-i}. Universal classes and classes with intersections have weak amalgamation.

\begin{defin}\label{w-ap} $\K$ has \emph{weak amalgamation} if whenever $\gtp(a_1/ M, N_1) = \gtp(a_2/M, N_2)$ there are $N_1' \lea N_1$  and $N_2 \lea N_3$  such that $\{a_1\} \cup M \subseteq N_1'$ and $f: N_1' \xrightarrow[M]{} N_3$ is a $\K$-embedding with $f(a_1)=a_2$. 
\end{defin}

\begin{lem}\label{m-lemma}
Let $\K$ be an AEC and let $\lambda \ge \LS(\K)$. Assume  $\K$ has amalgamation in $\lambda$, $\K$ is stable in $\lambda$, and  splitting is continuous in $\lambda$. If $\K$ is categorical in $\lambda^+$, $\K$ is $\lambda$-tame and has weak amalgamation, then $\K_{\geq \lambda} $ has amalgamation  and $\K$ has arbitrarily large models.
\end{lem}
\begin{proof}
As before we can assume without loss of generality that $\K$ has joint embedding and no maximal models in $\lambda$. Moreover, we assume that $\K_{<\lambda} = \emptyset$. 

Let  $\s=( \K, \dnf , \gS^{bs})$ be given by:
\begin{itemize}
\item For $M \in \K_{\lambda^+}$, $\gS^{bs}(M)= \gS^{na}(M)$.
\item For $M, N, R \in \K_{\lambda^+}$ we define:
$a \dnf_{M}^{R} N$ if and only if   $M \lea N \lea R$, $a\in |R| - |N|$ and there is $M' \in \K_\lambda$ with $M' \lea M$ such that for every $N' \in \K_{\lambda}$ with $M' \lea N' \lea N$ there is $M_0' \in \K_\lambda$ such that  $M_0' \lea M'$, $M'$ is universal over $M_0'$ and  $\gtp(a/N, R)\rest N'$  does not split over $M_0'$.  We say that $\gtp(a/N, R)$ does not $\lambda^+$-fork over $M$.

\end{itemize}

\underline{Claim}: $\s=( \K, \dnf , \gS^{bs})$ is a w-good $\lambda^+$-frame with density.

  \underline{Proof of Claim}: This frame was first considered in \cite[4.2,3.8]{sebastien-s-first-good-frame-construction-paper} under different assumptions, we show that everything still goes through in our  setting. First observe that all the models in $\lambda^+$ are $\lambda^+$-model-homogeneous  by Fact \ref{model-h} so we can apply the results of \cite[\S 4, 5]{sebastien-s-first-good-frame-construction-paper}. $\s$ is a pre-$\lambda^+$-frame by \cite[4.6]{sebastien-s-first-good-frame-construction-paper}, $\K_{\lambda^+}$ has no maximal models by Theorem \ref{++} and has joint embedding by categoricity in $\lambda^+$, $\s$  has: density by \cite[4.9]{sebastien-s-first-good-frame-construction-paper}, uniqueness by $\lambda$-tameness and \cite[5.3]{sebastien-s-first-good-frame-construction-paper}, and transitivity by \cite[4.10]{sebastien-s-first-good-frame-construction-paper}. We show continuity,  existence of non-forking extensions, and amalgamation in $\lambda^+$ as these are shown in \cite{sebastien-s-first-good-frame-construction-paper} under additional assumptions.

 \begin{itemize}

 \item \underline{Continuity}:  Let $\delta < \lambda^{++}$ be a limit ordinal which we may assume to be a regular cardinal, $\langle M_i : i < \delta \rangle $ be an increasing continuous chain in $ \K_{\lambda^+}$ and $p \in \gS^{na} (M_\delta)$ such that for every $i < \delta$,  $p\rest M_i$ does not $\lambda^+$-fork over $M_0$.    There is $M^* \in \K_\lambda$ such that $M^* \lea M_\delta$ and  $p$ does not split over $M^*$ by Fact \ref{w2}. There are two cases to consider:
 
 \underline{Case 1}: $\delta=\lambda^+$. Then there is $i < \lambda^+$, such that $M^* \lea M_i$. Hence $p$ does not $\lambda^+$-fork over $M_i$ by \cite[4.8]{sebastien-s-first-good-frame-construction-paper}. Then by the assumption that $p\rest M_i$ does not $\lambda^+$-fork over $M_0$ and transitivity of $\s$, we have that $p$ does not $\lambda^+$-fork over $M_0$.

\underline{Case 2}: $\delta \leq \lambda$.  For each $i < \delta$, there is $M_0^i \in \K_\lambda$ such that $M_0^i \lea M_0$ and $p\rest M_i$ does not split over $M_0^i$  by \cite[4.8]{sebastien-s-first-good-frame-construction-paper}. Since $\delta \leq \lambda$, using stability in $\lambda$, monotonicity of splitting and that $M_0$ is $\lambda^+$-model-homogeneous, there are $M_{0,0} \lea  M_{0, 1} \in \K_\lambda$ such that $M_{0,1} \lea M_0$, $M_{0,1}$ is universal over $M_{0,0}$ and for every $i < \delta$, $p\rest M_i$ does not split over $M_{0,0}$. 

Let $N^*, N^{**} \lea M_\delta$ both in $\K_\lambda$ such that $N^{**}$ is universal over $N^*$ and $M_{0,0} \cup M^* \subseteq N^*$. Using stability in $\lambda$, monotonicity of splitting and that  the $M_i$'s are $\lambda^+$-model-homogeneous, one can build  $ \langle N_i : i \leq \delta \rangle$ in $\K_{\lambda}$  increasing continuous such that $N_0 =M_{0,1}$, $N_{i+1}$ is universal over $N_i$, $N_i \lea M_i$, $N^{**} \cap M_{i+1} \subseteq N_{i+1}$ and $p\rest N_i$ does not split over $M_{0,0}$.  Since splitting is continuous in $\lambda$ by assumption, $p\rest N_\delta$ does not split over $M_{0,0}$. 
 
 We show that $M_{0,1} \lea M_0$ witnesses that $p$ does not $\lambda^+$-fork over $M_0$.  Let $N' \in \K_\lambda$ with $M_{0,1} \lea N' \lea M_\delta$ and $M_0' = M_{0,0}$, we show that $p\rest N'$ does not split over $M_{0,0}$. Let $L \in \K_\lambda$ be the structure obtained by applying the Löwenheim-Skolem-Tarski axiom to $N_\delta \cup N^{**} \cup N'$ in $M_\delta$. By monotonicity of splitting $p\rest L$ does not split over $N^*$. Since  $p\rest N_\delta$ does not split over $M_{0,0}$ and  $N_\delta$ is universal over $N^*$ because $N^{**} \lea N_\delta$, then  $p\rest L$ does not split over $M_{0,0}$ by weak transitivity. Therefore $p\rest N'$ does not split over $M_{0,0}$ by monotonicity of splitting.

\item  \underline{Existence of non-forking extension}: Let $M \lea N$ both in $\K_{\lambda^+}$ and $p \in \gS^{na}(M)$. There is $M^* \in \K_\lambda$ such that $M^* \lea M$ and  $p$ does not split over $M^*$ by Fact \ref{w2}. First build an increasing continuous chain $ \langle M_i : i <\lambda^+ \rangle$ in $\K_\lambda$ with   $M_i \lea M$  for all $i < \lambda^+$, $M_{0}$ is universal over $M^*$ and $M =\bigcup_{i < \lambda^+} M_i$ . 

Let $\{ n_i : i <\lambda^+ \}$ be an enumeration of $N$. We build, as in Theorem \ref{++} using that $N$ is $\lambda^+$-model-homogeneous, an increasing continuous chain $ \langle N_i : i < \lambda^+ \rangle$ in $\K_\lambda$ and $ \langle p_i :  i < \lambda^+ \rangle$ a coherent sequence of types such that:

\begin{enumerate}
\item  $N_0 = M_{0}$ and $p_0 =  p \rest M_{0}$;
\item for every $i <\lambda^+$, $n_i \in N_{i+1}$, $N_i \lea N$ and $N_{i+1}$ is universal over $N_i$;
    \item for every $i < \lambda^+$,  $p_{i} \in \gS^{na}(N_{i})$ does not split over $M^*$;
 \item if $i< j< \lambda^+$, then $p_i\leq p_j$;
\item for every $j < \lambda^+$, $\langle p_i :  i<j\rangle$ is coherent . 
\end{enumerate}

 Let  $p_{\lambda^+}\in \gS^{na}(\bigcup_{i < \lambda^+} N_i) =\gS^{na}(N)$ be an upper bound of the coherent sequence $\langle p_i :  i <\lambda^+\rangle$. We show that $p_{\lambda^+} \geq p$ and that $p_{\lambda^+}$ does not $\lambda^+$-fork over $M$.

 We show that for every $i < \lambda^+$, $p_i\rest M_i=p\rest M_i$. This is enough to show that   $p_{\lambda^+} \geq p$ as $\K$ is $\lambda$-tame. Let $i < \lambda^+$. Observe $p_i\rest M_i$ does not split over $M^*$ by condition (3), $p\rest M_i$ does not split over $M^*$ by monotonicity of splitting, $(p_i\rest M_i)\rest M_0 = p_0 = p\rest M_0= (p\rest M_i)\rest M_0 $ by conditions (1), (4) and $M_0$ is universal over $M^*$, then $p_i\rest M_i=p\rest M_i$ by weak uniqueness.

 We show that $M_{0} \lea M$ witnesses that $p$ does not $\lambda^+$-fork over $M$.  Let $N' \in \K_\lambda$ with $M_{0} \lea N' \lea N$ and $M_0' = M^*$. Observe that $p\rest N'$ does not split over $M^*$ by condition (3) and monotonicity of splitting. 
 \item \underline{Amalgamation in $\lambda^+$}: It follows from density and existence of non-forking extension of $\s$ and weak amalgamation by \cite[4.16]{universal-class-part-i}.   $\dagger_{\text{Claim}}$
 \end{itemize}

Since $\K$ has weak amalgamation, is $\lambda$-tame and $\s$ is a w-good $\lambda$-frame with density, one can show that  $\K$ has a $[\lambda^+, \infty)$-w-good frame with density following the arguments of \cite[4.16]{universal-class-part-i} and \cite[3.24]{m1} (see also \cite{extendingframes}). In particular, $\K_{\geq \lambda}$ has  amalgamation and $\K$ has arbitrarily large models. 
\end{proof}

\begin{thm}\label{trans}
Let $\K$ be an AEC with weak amalgamation and let $\lambda \ge \LS(\K)$ be such that $\K$ is $\lambda$-tame. Assume  $\K$ has amalgamation in $\lambda$, $\K$ is stable in $\lambda$, and  splitting is continuous in $\lambda$. If $\K$ is categorical in $\lambda$ and $\lambda^+$, then $\K$ is categorical in all $\mu \ge \lambda$.
\end{thm}
\begin{proof}
$\K_{\geq \lambda}$ has amalgamation and $\K$ has arbitrarily large models by Lemma \ref{m-lemma}. Therefore, $\K$ is categorical in all $\mu \ge \lambda$ by  Grossberg-VanDieren results \cite[5.2, 6.3]{tamenessthree}.
\end{proof}

A simpler result to state is the following. 
\begin{cor}\label{uni}
Let $\K$ be a universal class and let $\lambda \ge \LS(\K)$. Assume  $\K$ has amalgamation in $\lambda$ and $\K$ is stable in $\lambda$. If $\K$ is categorical in $\lambda$ and $\lambda^+$, then $\K$ is categorical in all $\mu \ge \lambda$.
\end{cor}
\begin{proof}
Universal classes are $(<\aleph_0)$-tame \cite[3.7]{universal-class-part-i}. Therefore, they are $(<\lambda^+,\lambda)$-local and $\lambda$-tame (see for example \cite[2.5]{may}). Then splitting is continuous in $\lambda$ by Lemma \ref{loc->con}. As universal classes have  weak amalgamation, the result follows from Theorem \ref{trans}.
\end{proof}

  In  \cite{bk}, it was shown that for every $k\geq 2$ there are $\mathbb L_{\omega_1,\omega}$ sentences that have amalgamation in all cardinals and arbitrarily large models, are categorical in $\aleph_{k-2}$ and $\aleph_{k-3}$, and are stable in $\aleph_{k-3}$; but they are not $(\aleph_{k-3}, \aleph_{k-2})$-tame and not categorical on a tail of cardinals.  Thus tameness is not necessary in Lemma \ref{m-lemma} and  tameness or continuity of splitting are necessary in Theorem \ref{trans}. So following a referee suggestion we ask:

  \begin{question}\
  \begin{itemize}
      \item Can one prove Lemma \ref{m-lemma} without tameness or even continuity of splitting?
  \item What are the roles of the tameness and continuity of splitting assumptions in Theorem \ref{trans}?
   \end{itemize}
  \end{question}

  It is worth pointing out that it is not known if the examples of \cite{bk} have continuity of splitting.

\section{Getting stability in $\lambda$}

We show that stability is necessary to construct a model of cardinality $\lambda^{++}$ under mild cardinal arithmetic assumptions, categoricity in two succesive cardinals and almost stability for $\lambda^+$-minimal types. Observe that the main results in this section do not require that splitting is continuous.

\subsection{Almost stable in $\lambda$}

The following theorem is \cite[\S VI.2.11]{shelahaecbook} and \cite[2.5.8]{JaSh875}. We include the details for the sake of completness.

\begin{fact}\label{inev}
    Suppose $\K$ has amalgamation and no maximal model in $\lambda$. Let $\gS_*$ be a $\lea$-type kind and hereditary. Suppose that for all $M\in \K_\lambda$ there is $\Gamma_M\subseteq \gS_*(M)$ such that $|\Gamma_M|\leq \lambda^+$ and $\Gamma_M$ is $\gS_*$-inevitable. Then there is a model saturated for $\gS_*$-types in $\lambda^+$ above $\lambda$. In particular, for all $M \in \K_\lambda$, $| \gS_*(M) | \leq \lambda^+$. 
\end{fact}
\begin{proof}
    Fix a bijection $g:\lambda^+\times \lambda^+\to \lambda^+$. We build $\langle M_i:i<\lambda^+\rangle$ and $\langle p_{i,j}:i,j<\lambda^+\rangle$ such that: \begin{enumerate}
        \item $M_i\in \K_\lambda$ for all $i<\lambda^+$;
        \item $\langle M_i:i<\lambda^+\rangle$ is increasing and continuous;
        \item $\{p_{i,j} :j<\lambda^+\}=\Gamma_{M_i}$ for all $i<\lambda^+$;
        \item $M_{i+1}$ realizes $p_{g(\epsilon)}$, where $\epsilon$ is the least such that $g(\epsilon)=(\alpha,\beta)$, $\alpha\leq i$, and $p_{\alpha,\beta}$ is not realized in $M_i$.
    \end{enumerate}
    We now claim that $M_{\lambda^+}:=\bigcup_{i<\lambda^+} M_i$ is saturated for $\gS_*$-types above $\lambda$. It suffices to show that for any $M_0<_\K N\in \K_\lambda$, $a\in |N|-|M_0|$, $p=\type(a/M_0,N) \in \gS_*(M_0)$, $p$  is realized in $M_{\lambda^+}$.

     We build: $\langle N_i :i<\lambda^+\rangle$, $\langle \alpha_i : i<\lambda^+\rangle$ and $f_i:M_{\alpha_i}\to N_i$ such that:
    \begin{enumerate}
        \item $N_i\in \K_\lambda$ for all $i<\lambda^+$;
        \item $\alpha_i<\lambda^+$ for all $i<\lambda^+$;
        \item $\langle N_i:i<\lambda^+\rangle$ is increasing and continuous;
        \item $\langle f_i:i<\lambda^+\rangle$ is increasing and continuous;
        \item $\langle \alpha_i : i<\lambda^+\rangle$ is increasing and continuous;
        \item $N_0=N$;
        \item $\alpha_0=0$;
        \item $f_0=id_{M_0}$;
        \item $\|N_i|-|f_i[M_{\alpha_i}]\|\geq 1$;
        \item For each $i<\lambda^+$ there is $b\in |M_{\alpha_{i+1}}|-|M_{\alpha_i}|$ such that $f_{i+1}(b)\in |N_{i}|$.
    \end{enumerate}

    We carry out the construction by induction on $i<\lambda^{+}$. The base is clear.  At successor $i+1$, if $a\in  f_i[M_{\alpha_i}]$, then already $M_{\alpha_i}$ realizes $\type(a/M_0,N)$, and we are done. We will prove that this must happen for some $i$.
    
    Otherwise we continue the construction. Since $\|N_i|-|f[M_{\alpha_i}] \|\geq 1$ and $\type(a/f_i[M_{\alpha_i}],N_i)\in \gS_*(f_i[M_{\alpha_i}])$ because $S_*$ is hereditary, by inevitability there is $b\in |N_i|-|f_i[M_{\alpha_i}]|$ such that $\type(b/f_i[M_{\alpha_i}],N_i)=f_i(p) $ for some $p\in \Gamma_{M_{\alpha_i}}$. (Why? note that the image of an $\gS_*$-inevitable set remains $\gS_*$-inevitable, so $\{f_i(q):q\in \Gamma_{M_{\alpha_i}}\}$ is $\gS_*$-inevitable.) There is $\alpha_{i+1}$ such that $M_{\alpha_{i+1}}$ realizes $p$ by condition (4) of the construction of $\langle M_i:i<\lambda^+\rangle$, so we  can find $f_{i+1}$ and $N_{i+1}$ such that 
    \begin{equation*}
        \begin{tikzcd}
        N_i \ar[r] & N_{i+1}\\
        M_{\alpha_i} \ar{u}{f_i} \ar[r]&  M_{\alpha_{i+1}} \ar{u}{f_{i+1}}
        \end{tikzcd}
    \end{equation*}
    commutes with $b=f_{i+1}(c)$ for some $c\in |M_{\alpha_{i+1}}|$ and $\| N_{i+1}| -|f_{i+1}[M_{\alpha_{i+1}}]\|\geq 1$

When $i$ is a limit, choose $N_i := \bigcup_{j<i}N_j$ and $f_i:=\bigcup_{j<i} f_j$.  Observe $\|N_i|-|f[M_{\alpha_i}]\|\geq 1$ as otherwise $a \in   f_j[M_{\alpha_j}]$ for $j <i$ and we would have stopped the construction.

    Finally let $N_{\lambda^+}:=\bigcup_{i<\lambda^+}N_i$, $f:=\bigcup_{i<\lambda^+} f_i$, and $N^*:=f[M_{\lambda^+}]$. Now $\langle f_i[M_{\alpha_i}]:i<\lambda^+\rangle$ and $\langle N_i \cap N^*:i<\lambda^+\rangle$ are two resolutions of $N^*$, so $f_i[M_{\alpha_i}]= N_i\cap N^*$ for some $i$. Then $f_i[M_{\alpha_i}]\subseteq N_i\cap f_{i+1}[M_{\alpha_{i+1}}]\subseteq N_i\cap N^* =f_i[M_{\alpha_i}]$, but this contradicts condition (10) of $\langle N_i:i<\lambda^+\rangle$. 

    Thus, the construction of $\langle N_i:i<\lambda^+\rangle$ and $\langle f_i:i<\lambda^+\rangle $ is not possible, so it must be that for some $i<\lambda^+$, $a\in f_i[M_{\alpha_i}]$. This shows that $M_{\lambda^+}$ realizes $p$. \end{proof}
    

\begin{thm}\label{almost-stability}
    Assume $2^{\lambda} < 2^{\lambda^+}$. Suppose $\K$ has amalgamation and no maximal model in $\lambda$, categoricity in $\lambda$ and $\lambda^+$, and $|\gS^{\neg \lambda^+-min}(M)|\leq \lambda^+$  for the unique model $M\in \K_\lambda$. Then $\K$ is almost stable in $\lambda$.
\end{thm}
\begin{proof}
   Assume for the sake of contradiction that $|\gS(M)|>\lambda^+$. Then $|\gS^{\lambda^+-min}(M)|>\lambda^+$ as $|\gS^{\neg \lambda^+-min}(M)|\leq \lambda^+$. Since $|\gS^{\lambda^+-min}(M)|\geq \lambda^{++}$, no subset of $\gS^{\lambda^+-min}(M)$ of size $\leq \lambda^+$ is $\gS^{\lambda^+-min}$-inevitable by categoricity in $\lambda$ and Fact \ref{inev}. We build $\langle M_\eta:\eta\in {}^{<\lambda^+}2\rangle$ and $\langle \Gamma_\eta:\eta\in {}^{<\lambda^+}2\rangle$ such that:
    \begin{enumerate}
        \item $M_\eta\in \K_\lambda$ for all $\eta\in {}^{<\lambda^+}2$;
        \item $\Gamma_\eta\subseteq \bigcup_{j\leq i}\gS^{\lambda^+-min}(M_{\eta\restriction i})$ for all $\eta\in {}^i 2$, $i<\lambda^+$;
        \item $|\Gamma_\eta |\leq \lambda^+$ for all $\eta \in {}^{<\lambda^+}2$;
        \item $M_\eta$ omits all types in $\Gamma_\eta$ for all $\eta\in {}^{<\lambda^+}2$.
        \item If $\eta<\nu\in{}^{<\lambda^+} 2$, then $M_\eta\lea M_\nu$ and $\Gamma_\eta\subseteq \Gamma_\nu$;
        \item $\lambda+i \leq |M_\eta|\leq \lambda +\lambda\cdot i$ for $\eta\in {}^i 2$, $i<\lambda^+$;
        \item For all $\eta$, $M_{\eta\concat \ell}$ realizes a type over $M_\eta$ from $\Gamma_{\eta\concat (1-\ell)}$ for $\ell=0,1$;
        \item For all $\eta$, $\{ \type(a/M_\eta, M_{\eta\concat 1- \ell})\in \gS^{\lambda^+-min}(M_\eta) : a \in M_{\eta\concat (1- \ell)}\} \subseteq \Gamma_{\eta \concat \ell} $ for $\ell=0,1$. 
    \end{enumerate}
    \fbox{Construction} We build everything by induction on the $i$, the length of $\eta$. Let $M_{\langle\rangle}$ be the unique model in $\K_\lambda$ and $\Gamma_{\langle\rangle}:=\emptyset$. At limits let $M_\eta:=\bigcup_{j<i} M_{\eta\restriction j}$ and $\Gamma_\eta:=\bigcup_{j<i} \Gamma_{\eta\restriction j}$. At successor $i+1$, let $M_{\eta\concat 0}$ be any $\lea$-extension of $M_\eta$ such that:
    \begin{enumerate}
       
        \item $M_{\eta\concat 0}$ omits $$
            \Gamma'_\eta:=\bigcup_{j<i} \{q\in \gS^{na}(M_\eta) : q\restriction M_{\eta\restriction j}\in \Gamma_{\eta \restriction j }\}.     $$
        \item some $a\in |M_{\eta\concat 0}|$ satisfies that $\type(a/M_\eta, M_{\eta\concat 0})\in \gS^{\lambda^+-min}(M_\eta)$. 
    \end{enumerate}

    Note that $M_{\eta\concat 0}$ exists as we can omit any set of $\lambda^+$-minimal types of size $\leq \lambda^+$ while realizing at least one type that is $\lambda^+$-minimal as no set of $\lambda^+$ types is $\gS^{\lambda^+-min}$-inevitable. This is possible as each type in $\Gamma_{\eta \restriction j}$ has $\leq\lambda^+$ extensions to $\gS(M_\eta)$ so $|\Gamma'_\eta|\leq \lambda^+$ and every type in $\Gamma'_\eta$ is $\lambda^+$-minimal. Arrange that $|M_{\eta\concat 0}|$ is an ordinal $|M_\eta|+\kappa$ for some $1\leq \kappa \leq \lambda$. By induction $\lambda+i\leq |M_\eta|\leq \lambda+\lambda\cdot i$, so we obtain $\lambda+(i+1)\leq |M_{\eta\concat 0}|\leq \lambda+\lambda\cdot i + \lambda=\lambda+\lambda\cdot (i+1)$.
    
    Let $$\Gamma'_{\eta\concat 1}:=\Gamma'_\eta \cup \{\type(a/M_\eta,M_{\eta\concat 0}): \type(a/M_\eta,M_{\eta\concat 0})\in \gS^{\lambda^+-min}(M_\eta)\}$$

    and $M_{\eta\concat 1}$ be any $\K$-extension of $M_\eta$ such that:
    \begin{enumerate}
        \item $\lambda+i\leq |M_{\eta\concat 1}|\leq \lambda+\lambda\cdot (i+1)$;
        \item $M_{\eta\concat 1}$ omits $\Gamma'_{\eta\concat 1}$;
        \item some $a\in |M_{\eta\concat 1}|$ satisfies that $\type(a/M_\eta, M_{\eta\concat 1})\in \gS^{\lambda^+-min}(M_\eta)$.
    \end{enumerate}

$M_{\eta\concat 1}$ exists because $\Gamma'_{\eta\concat 1}$ is not $\gS^{\lambda^+-min}$-inevitable. One can check that the other requirements are satisfied as for $M_{\eta\concat 0}$.

Finally, let 
$$\Gamma_{\eta\concat 0}:=\Gamma_\eta \cup \{\type(a/M_\eta,M_{\eta\concat 1}) : \type(a/M_\eta,M_{\eta\concat 1})\in \gS^{\lambda^+-min}(M_\eta)\},$$
and 
$$\Gamma_{\eta\concat 1}:=\Gamma_\eta \cup  \{\type(a/M_\eta,M_{\eta\concat 0}): \type(a/M_\eta,M_{\eta\concat 0})\in \gS^{\lambda^+-min}(M_\eta)\}.$$

We only check requirement (4) of the construction as the others are easy to check.  We show that $M_{\eta\concat 0}$ omits $\Gamma_{\eta\concat 0}$. Let $p\in \Gamma_{\eta\concat 0}$. Assume for the sake of contradiction that $p$ is realized by $a\in M_{\eta\concat 0}$. $p$ cannot be of the form $\gtp(b/M_\eta, M_{\eta\concat 1})$ as $M_{\eta\concat 1}$ omits every $\lambda^+$-minimal type over $M_\eta$ realized in $M_{\eta\concat 0}$ by (2) of the construction of $M_{\eta\concat 1}$. So $p\in \Gamma_\eta$, then $M_{\eta\concat 0}$ omitted all non-algebraic extensions of $p$ as they are in $\Gamma_\eta'$, and any algebraic extension of $p$ cannot be realized since it must lie in $M_\eta$, but $M_\eta$ omits $\Gamma_\eta$ by induction hypothesis. Similarly $M_{\eta\concat 1}$ omits $\Gamma_{\eta\concat 1}$.

\fbox{Enough} Let $C:=\{\delta<\lambda^+: \lambda+\lambda\cdot \delta=\delta=\lambda+\delta\}$. Note that $C$ is a club. By Fact \ref{Ulam} there are disjoint stationary sets $S_\gamma\subseteq \lambda^+$ such that $\Phi_{\lambda^+}(S_\gamma)$ holds for all $\gamma<\lambda^+$.

We denote the zero sequence in ${}^{\lambda^+}2$ by $\bar 0$. For $\delta\in C$, and $\eta\in  {}^\delta 2$ and $h:\delta\to \delta$, define

\[
F(\eta,h):= \begin{cases} 
      1 & \mbox{if }h:M_\eta\to M_{\bar 0\restriction \delta}\mbox{ and for some } \beta<\lambda^+ $ and $  g:M_{\eta\concat 0}\to M_{\bar 0 \restriction \beta} $ extending $h \\
      & $there are $ a\in |M_{\eta\concat 0}|-|M_\eta|, b\in |M_{\bar 0 \restriction \beta}|-|M_{\bar 0 \restriction \delta}|$ such that $\\
      &\begin{tikzcd}
          M_{\eta\concat 0} \ar{r}{g} & M_{\bar 0 \restriction \beta} \\
          M_\eta \ar{u} \ar{r}{h} & M_{\bar 0\restriction \delta} \ar[u]
    \end{tikzcd}

        $commutes with $ g(a)=b\\
\\
        
        & $ and $ \type(a/M_\eta, M_{\eta\concat 0}) \in \gS^{\lambda^+-min}(M_\eta) \\
 
      0 & \mbox{otherwise.} 
   \end{cases}
\]


For all $\gamma<\lambda^+$, by $\Phi_{\lambda^+}(S_\gamma)$ there is $g_\gamma:\lambda^+\to 2$ such that for all $\eta\in {}^{\lambda^+}2$ and $h:\lambda^+\to\lambda^+$.
$$
\{\delta\in S\gamma :  F(\eta\restriction\delta , h\restriction\delta)=g_\gamma(\delta)\}
$$
is stationary.

For each $X\subseteq\lambda^+$ define
\[\eta_X(\delta):= \begin{cases} 
      g_\gamma(\delta) & \mbox{if } \delta\in S_\gamma, \gamma\in X,\\
 
      0 & \mbox{otherwise.} 
   \end{cases}
\]

Since $\Ii(\K, \lambda^+)=1$, the following claim would give us a contradiction. 

\underline{Claim}:  For $X\neq \emptyset \subseteq \lambda^+$, there is no $h:M_{\eta_X}\cong M_{\bar 0}$.

\underline{Proof of Claim.} Assume there are such $X$ and $h$. Let $\gamma\in X$.  Then $$D:=\{\delta<\lambda^+ : h\restriction \delta:\delta\to \delta\}$$ is a club.  Let $S'_\gamma =\{\delta\in S\gamma :  F(\eta_X \restriction\delta , h\restriction\delta)=g_\gamma(\delta)\}
$ be the stationary set obtained from $\eta_X, h$.

Let $\delta\in S'_\gamma\cap D \cap C$. Observe that for all $\eta\in {}^\delta 2$, $\delta=\lambda+\delta \leq|M_\eta|\leq \lambda+\lambda \cdot \delta=\delta$, i.e. $|M_\eta|=\delta$. Since $h\restriction\delta:\delta\to \delta$, $h$ is a $\K$-embedding from $M_{\eta_X\restriction \delta}$ to $M_{\bar 0\restriction \delta}$.

We divide the proof into two cases:

\textbf{Case 1:} $\eta_X(\delta)=1$. Then $g_\gamma(\delta)=F(\eta_X \restriction\delta , h\restriction\delta)=1$, so the following diagram commutes
\begin{equation}
\begin{tikzcd}
    M_{(\eta_X\restriction \delta) \concat 0} \ar{r}{g}&  M_{\bar 0\restriction \beta}\\
          M_{\eta_X\restriction \delta} \ar{u} \ar{r}{h}& M_{\bar 0 \restriction \delta}\ar[u]
    \end{tikzcd}
\end{equation}
with 
\begin{equation}\label{dag1} \tag{$\dagger_1$}
    g(a)=b
\end{equation}
for some $\delta<\beta<\lambda^+$, $a\in |M_{(\eta_X\restriction \delta)\concat 0}|-|M_{\eta_X\restriction\delta}|$ and $b\in|M_{\bar 0\restriction \beta} |-|M_{\bar 0\restriction \delta}|$ such that $$\type(a/M_{\eta_X \restriction\delta}, M_{(\eta_X\restriction \delta)\concat 0})\in \gS^{\lambda^+-min}(M_{\eta_X\restriction\delta}).$$

Since $h:M_{\eta_X}\cong M_{\bar 0}$ and $b\in |M_{\bar 0 \restriction \beta}|$, there are $\delta<\beta'<\lambda^+$ and $c\in |M_{\eta_X\restriction \beta'}|$ such that the following diagram commutes:
\begin{equation}
\begin{tikzcd}
 M_{\eta_X\restriction\delta} \ar[r,"h"] \ar[d,"id"] & M_{\bar 0\restriction \delta}   \ar[d,"id"]\\
 M_{\eta_X\restriction\beta'} \ar[r,"h"] & M_{\bar 0 \restriction \beta'}
\end{tikzcd}
\end{equation}
with
\begin{equation}\label{dag2}\tag{$\dagger_2$}
    h(c)=b
\end{equation}
 Note that $c\notin |M_{\eta_X\restriction \delta}|$ since $h(c)=b\notin |M_{\bar 0\restriction \delta}|$. Without loss of generality assume $\beta=\beta'$ as we can take them arbitrarily large as long as $M_{\bar 0\restriction \beta}$ contains $b$, $h[M_{\eta_X\restriction \beta'}]$ and $g[M_{(\eta_X\restriction \delta)\concat 0}]$.

Now we put the two diagrams together:
\begin{equation}
\begin{tikzcd}
M_{(\eta_X\restriction \delta)\concat 0} \ar{rrd}{g}& &\\
          M_{\eta_X\restriction \delta} \ar{d} \ar{u} \ar{r}{h}& M_{\bar 0 \restriction \delta}  \ar{r} &M_{\bar 0\restriction \beta}\\
          M_{\eta_X\restriction\beta} \ar[rru,"h"] & &
\end{tikzcd}
\end{equation}
Now the outer diagram
\begin{equation}
\begin{tikzcd}
 M_{(\eta_X\restriction \delta)\concat 0} \ar[r,"g"] & M_{\bar 0 \restriction \beta}\\
 M_{\eta_X\restriction\delta} \ar[u, "id"] \ar[r, "id"] & M_{\eta_X \restriction\beta}\ar[u,"h"]
\end{tikzcd}
\end{equation}
commutes with $g (a)=b=h(c)$ by (\ref{dag1}) and (\ref{dag2}), so $$\type(a/M_{\eta_X\restriction\delta}, M_{(\eta_X\restriction\delta)\concat 0})=\type(c/M_{\eta_X\restriction\delta}, M_{\eta_X\restriction\beta}).$$
This is impossible, since $\type(a/M_{\eta_X\restriction\delta}, M_{(\eta_X\restriction\delta)\concat 0})\in \Gamma_{(\eta_X\restriction \delta)\concat 1}$ by requirement (8), and $\Gamma_{(\eta_X\restriction \delta)\concat 1}\subseteq \Gamma_{\eta_X\restriction\beta}$ is omitted by $M_{\eta_X\restriction\beta}$ by requirements (4) and (5). This finishes \textbf{Case 1}.\\
\textbf{Case 2:} $\eta_X(\delta)=g_\gamma(\delta)=F(\eta_X\restriction\delta, h\restriction\delta)=0$. We now show that $$F(\eta_X\restriction\delta, h\restriction\delta)=1$$ so that this case is not possible.

Let $a\in |M_{(\eta_X\restriction \delta)\concat 0}|-|M_{\eta_X\restriction \delta}|$ such that $\type(a/M_{\eta_X\restriction \delta},M_{(\eta_X\restriction \delta)\concat 0} )\in \gS^{\lambda^+-min}(M_{\eta_X\restriction \delta})$. We can find such $a$ by the condition (7) of the construction. Since $b:=h(a)\in |M_{\bar 0 }|$, find $\delta<\beta<\lambda^+$ such that

\begin{equation*}
    \begin{tikzcd}
        M_{(\eta_X\restriction \delta)\concat 0} \ar[r,"h"] & M_{\bar 0 \restriction \beta} \\
        M_{\eta_X\restriction \delta} \ar[u]\ar[r,"h"] &M_{\bar 0\restriction\delta} \ar{u}
    \end{tikzcd}
\end{equation*}
commutes with $h(a)=b$ and $b\in |M_{\bar 0\restriction \beta}|-|M_{\bar 0 \restriction \delta}|$ (since $a\notin |M_{\eta_X\restriction \delta}|)$. Thus $F(\eta_X\restriction\delta,h\restriction\delta)=1$, which contradicts the assumption of Case 2. $\dagger_{\text{Claim}}$

\end{proof}

\subsection{Stable in $\lambda$} We turn our attention to obtain stability in $\lambda$.

\begin{lem} [{\cite[6.3]{sh430}}] \label{tree}
If $\lambda^+<2^\lambda$, then there is a tree with $\leq\lambda$ nodes and $\kappa\leq \lambda$ levels with at least $\lambda^{++}$ branches of length $\kappa$.
\end{lem}
\begin{proof}
    Let $\kappa$ be the least cardinal such that $2^\kappa>\lambda^+$. Consider the tree ${}^{<\kappa} 2$. If $2^{<\kappa}\leq \lambda$ this tree is enough. Indeed, its set of branches of length $\kappa$ is just ${}^\kappa 2$, which is of cardinality $2^\kappa>\lambda^+$. 
    
    Now suppose $2^{<\kappa}>\lambda$. Since $2^\lambda>\lambda^+$, $\kappa\leq \lambda$. Then $2^{<\kappa}=\lambda^+$ by the assumptions that $2^{<\kappa}>\lambda$ and that $\kappa$ is minimal.  Write $2^{<\kappa}=\bigcup_{i<\lambda^+} B_i$, $B_i$ increasing with $i$, $|B_i|\leq \lambda$. For each $\eta\in {}^\kappa 2$ and each $\alpha<\kappa$, $\eta\restriction \alpha\in B_i$ for some $i$. Then there is $j(\eta)$ such that $\eta\restriction \alpha\in B_{j(\eta)}$ happens for cofinally many $\alpha<\kappa$. As for each $\alpha<\kappa$ there is $k_\alpha$ such that $\eta\restriction \alpha \in B_{k_\alpha}$. $\sup \{k_\alpha:\alpha<\kappa\} <\lambda^+$, so take $j(\eta)$ to be this supremum. Consider the map $\eta\mapsto j(\eta)$ from ${}^\kappa 2$ to $\lambda^+$. By the pigeonhole principle there is $j^*$ such that $|\{\eta\in {}^{<\kappa}2 : j(\eta)<j^*\}|\geq \lambda^{++}$. Note that $\{\eta\in {}^{\kappa}2 : j(\eta)<j^*\}$ is the set of branches of length $\kappa$ of the tree $T:=\{ \eta\restriction \alpha : \alpha<\ell(\eta)$, $\eta\in B_{j^*}\}$. Moreover, $|T|\leq \kappa\cdot |B_{j^*}|\leq \lambda\cdot \lambda=\lambda$.
\end{proof}

We will use the next three fact due to Shelah \cite[\S VI]{shelahaecbook}. Recall given  $\| M \| =\lambda$ and  $p\in \gS^{na}(M)$, $p$ has the \emph{extension property} if for every $N \in \K_\lambda$, if $M \lea N$ then there is $q \in \gS^{na}(N)$ extending $p$. 

\begin{fact}[{\cite[VI.2.5.(1),(3)]{shelahaecbook}}] \label{many-t}
    If $M\in\K_\lambda$, and there is no minimal type above $p\in \gS^{na}(M)$, then there is $N\in \K_\lambda$ such that $p$ has $ \lambda^+$ extensions to $\gS^{na}(N)$ and $p$ has the extension property. Note that above these extensions there are not minimal types either. 
\end{fact}

\begin{fact}\label{5.3}({\cite[VI.5.3(1)]{shelahaecbook}})
   Suppose $2^\lambda<2^{\lambda^+}$. Assume that $\K$ is categorical in $\lambda$ and  $\lambda^+$ and that $\K_{\lambda^{++}}\neq \emptyset$. If there is an minimal type over (the unique) $M\in \K_\lambda$, then there is an inevitable one.
\end{fact}

\begin{fact}[{\cite[VI.5.8(1)]{shelahaecbook}}]\label{5.8} 
    Assume $\K$ is categorical in $\lambda$, has amalgamation in $\lambda$, and has a model in $\lambda^{++}$. If there is an inevitable type over $M\in\K_\lambda$, then $\K$ is stable in $\lambda$.
\end{fact}

We obtain the main result of this section which is the forward direction of result mentioned in the abstract.

\begin{thm}\label{main-2} Suppose $\lambda^+ < 2^{\lambda} < 2^{\lambda^+}$.
 Assume $\K$ is categorical in $\lambda$ and $\lambda^+$, $\K_{\lambda^{++}} \neq \emptyset$ and $|\gS^{\neg\lambda^+-min}(M)|\leq \lambda^+$ for the unique model $M\in \K_\lambda$. Then $\K$ is stable in $\lambda$.
\end{thm}

\begin{proof} We show that there is a minimal type in $\K_\lambda$.  This is enough as it implies the existence of an inevitable minimal type by Fact \ref{5.3}, which in turn implies that $\K$ is  stable in $\lambda$ by Fact \ref{5.8}. 
    
    Assume for the sake of contradiction  that there is not a minimal type in $\K_\lambda$. Build $\langle M_i : i<\kappa\rangle$ and $\langle p_\eta:\eta\in T\rangle$, where $T$ is the tree from Lemma \ref{tree} which exists as $\lambda^+ < 2^{\lambda}$, such that:
\begin{enumerate}
    \item $M_i\in \K_\lambda$ for all $i<\kappa$;
    \item $\langle M_i: i<\kappa\rangle$ is increasing and continuous;
    \item For $\eta\leq\nu$, $p_\eta\leq p_\nu$;
    \item For all $\eta$ of rank $i$ and $\nu_0 \neq \nu_1\in T$, both of rank $i+1$ and extending $\eta$, $p_{\nu_0}\neq p_{\nu_1}$. 
    \item For all $i<\kappa$, $\langle p_{\eta\restriction \alpha}: \alpha<i\rangle$ is coherent.
\end{enumerate}
\fbox{Construction}  This is done by induction on the rank of $\eta\in T$. At stage $0$ let $p_{\langle\rangle}$ be any type (hence not minimal and having no minimal types above it). At successor stage, assume we have built the tree of types up to rank $i$ and we build the next level. Without loss of generality assume each $\eta$ of rank $i$ has $\lambda$ extensions $\{\eta^\alpha : \alpha<\lambda\}$ at the next level. We find $N_\eta$ and $\{ p_\eta^\alpha : \alpha<\lambda\}\subseteq \gS^{na}(N_\eta)$ distinct extensions of $p_\eta$ for each $\eta$, these exist by Fact \ref{many-t}. Amalgamate $N_\eta$ over $M_i$ for all $\eta$ to obtain $M_{i+1}$ and $f_\eta^i:N_\eta \to M_{i+1}$ such that $f_\eta \rest M_i = id_{M_i}$. For each $\eta$ and $\alpha < \lambda^+$, extend each $f_\eta^\alpha(p^\alpha_\eta)$ to $p_{\eta^\alpha}\in \gS^{na}(M_{i+1})$. Note that this is possible as any type above which there is no minimal type has the extension property. This finishes the successor case. At limit stage take directed colimits using Remark \ref{coh}. 


. \\
\fbox{Enough} Take $M:=\bigcup_{i<\kappa} M_i$, and $p_\eta$ be an upper bound of $\langle p_{\eta \restriction i} : i <\kappa \rangle$ for each branch (of length $\kappa$) of $T$. It is clear that $p_\eta \neq p_\nu$ if $\eta \neq \nu \in T$ by condition (4) of the construction. Therefore, $| \gS(M) |\geq \lambda^{++}$.  This is a contradiction as  $\K$ is almost stable in $\lambda$  by Theorem \ref{almost-stability}.
\end{proof}

A more natural result to state is the following. 
\begin{cor}\label{c-stable}
Suppose $\lambda^+ < 2^{\lambda} < 2^{\lambda^+}$.
 Assume $\K$ is categorical in $\lambda$ and $\lambda^+$, and is almost stable in $\lambda$. If $\K_{\lambda^{++}} \neq \emptyset$, then $\K$ is stable in $\lambda$.
\end{cor}

We are ready to prove the main equivalence of the paper.

\begin{thm}\label{*} Suppose $\lambda^+ < 2^{\lambda} < 2^{\lambda^+}$.
 Assume $\K$ is categorical in $\lambda$ and $\lambda^+$, $\K$ is almost stable for non-$\lambda^+$-minimal types in $\lambda$ and splitting is continuous in $\lambda$. The following are equivalent.
 
 \begin{enumerate}
 \item  $\K$ has a model in $\lambda^{++}$.
 \item $\K$ is stable in $\lambda$.
 \end{enumerate}
 \end{thm}
  \begin{proof}
  $(1) \Rightarrow (2)$: Theorem \ref{main-2}.
  
  $(2) \Rightarrow (1)$: $\K$ has amalgamation in $\lambda$ by Fact \ref{few_models->ap}, no maximal models in $\lambda$ by categoricity in $\lambda$ and $\K_{\lambda^+} \neq \emptyset$ and is stable in $\lambda$ by assumption. Moreover, splitting is continuous in $\lambda$ by assumption. Therefore $\K$ has a model in $\lambda^{++}$ by Theorem \ref{++}.  \end{proof}

  \begin{remark}
  Theorem \ref{main-2} can be used to replace the stability assumption in $\lambda$ of Theorem \ref{trans} by the weaker assumptions of: $\lambda^+ < 2^{\lambda} < 2^{\lambda^+}$,  $\K_{\lambda^{++}} \neq \emptyset$ and $|\gS^{\neg\lambda^+-min}(M)|\leq \lambda^+$ for the unique model $M\in \K_\lambda$.
  \end{remark}

\subsection{Further results} We obtain stability in $\lambda$ without the assumption of arbitrarily large models under different assumptions of those of Theorem \ref{main-2}. 

  In \cite[\S VI.4.2]{shelahaecbook}, it is shown that assuming $2^\lambda<2^{\lambda^+}$, then one of three statements about $\lambda$ hold, which we denote by (A), (B) and (C). For our purpose, there is no need to present them explicitly.

\begin{fact}[{\cite[\S VI.4.5(4)]{shelahaecbook}}]\label{minimal-dense-1} Assume $2^\lambda<2^{\lambda^+}$ and statement (A) holds for $\lambda$. If $\K$ has amalgamationin $\lambda$, $1\leq \Ii(\K,\lambda^+)<2^{\lambda^+}$, then for every non-algebraic type over any $M\in \K_\lambda$ there is a minimal type above it.
\end{fact}
\begin{fact}[{\cite[\S VI.4.9(2)]{shelahaecbook}}]\label{minimal-dense-2}
     Assume $2^\lambda<2^{\lambda^+}<2^{\lambda^{++}}$ and statement (B) or (C) hold for $\lambda$. If $\K$ has amalgamation in $\lambda$, is categorical in $\lambda^+$, $\K_{\lambda^{++}}\neq \emptyset$, and $|\gS(N)|<2^{\lambda^+}$ for the unique model in $\K_{\lambda^+}$, then for every non-algebraic type over any $M\in \K_\lambda$ there is a minimal type above it.
\end{fact}

\begin{lem}\label{main-2.1-shelah} Suppose that $2^{\lambda} < 2^{\lambda^+}<2^{\lambda^{++}}$.
 Assume $\K$ has amalgamation in $\lambda$ and is categorical in $\lambda^+$, and $|\gS(N)|<2^{\lambda^+}$ for the unique model $N\in \K_{\lambda^+}$. Then for every non-algebraic type over any $M\in \K_\lambda$ there is a minimal type above it.
\end{lem}
\begin{proof}
    If there are no minimal types, one can keep extending a non-algebraic type (not minimal) to a model in $\lambda^+$. Since that extension is non-algebraic, with categoricity in $\lambda^+$ one can show that $\K_{\lambda^{++}}\neq \emptyset$. If statement (A) holds for $\lambda$, we use Fact \ref{minimal-dense-1}. If (B) or (C) hold, and we use Fact \ref{minimal-dense-2}.
\end{proof}

\begin{thm}\label{stability-+}
    Suppose that $2^{\lambda} < 2^{\lambda^+}<2^{\lambda^{++}}$.
 Assume $\K$ is categorical in $\lambda$ and $\lambda^+$, $\K_{\lambda^{++}}\neq \emptyset,$ and $|\gS(N)|<2^{\lambda^+}$ for the unique model $N\in \K_{\lambda^+}$. Then $\K$ is stable in $\lambda$.
\end{thm}
\begin{proof}
    By Fact \ref{few_models->ap} $\K$ has amalgamation in $\lambda$. By Lemma \ref{main-2.1-shelah} there is a minimal type. This is enough as it implies the existence of an inevitable minimal type by Fact \ref{5.3}, which in turn implies that $\K$ is  stable in $\lambda$ by Fact \ref{5.8}. 
\end{proof}

\begin{remark}\label{prev-remark}
  Theorem \ref{stability-+} can be used to replace the stability assumption in $\lambda$ of Theorem \ref{trans} by the weaker assumptions of: $2^{\lambda} < 2^{\lambda^+}<2^{\lambda^{++}}$,  $\K_{\lambda^{++}} \neq \emptyset$ and  $|\gS(N)|<2^{\lambda^+}$ for the unique model $N\in \K_{\lambda^+}$.
  \end{remark}

A simple result to state following the idea of Remark \ref{prev-remark} is the following. 

\begin{cor} Assume $2^\lambda < 2^{\lambda^+} < 2^{\lambda^{++}}$. Let $\K$ be a universal class and let $\lambda \ge \LS(\K)$. If $\K$ is categorical in $\lambda$,  $\lambda^+$, $\K_{\lambda^{++}} \neq \emptyset$ and $|\gS(N)| < 2^{\lambda^+}$ for the unique $N \in \K_{\lambda^+}$, then $\K$ is categorical in all $\mu \ge \lambda$.
\end{cor}
\begin{proof}
The result follows from Theorem \ref{stability-+} and Corollary \ref{uni}. 
\end{proof}

We finish this section by observing that Fact \ref{inev} can be used to significantly simplify the proof of the following result due to Shelah. 
\begin{fact}[{\cite[VI.1.18, VI.1.20]{shelahaecbook}}]\label{additional_result}
    Assume $\K$ has amalgamation and no maximal models in $\lambda$. If $N\in \K_\lambda$, $\Gamma\subseteq \gS^{na}(N)$ and $|\Gamma|>\lambda^+$. Then we can find $N^*$ and $\langle N_i:i<\lambda^{++}\rangle$ such that:
    \begin{enumerate}
        \item $N\lea N^*<_\K N_i\in \K_\lambda$;
        \item For all $i\neq j<\lambda^{++}$ and $c_i\in |N_i|-|N^*|$, $c_j\in |N_j|-|N^*|$, $\type(c_i/N^*,N_i)\neq \type(c_j/N^*,N_j)$;
        \item there are $a_i\in |N_i|$ for $i<\lambda^{++}$ such that $\type(a_i/N,N_i)\in \Gamma$ is not realized in $N^*$, and these types are pairwise distinct; moreover $\type(a_i/N,N_i)$ is not realized in $N_j$ for $j<i$.
    \end{enumerate}
\end{fact}
\begin{proof} For every $M \in \K_\lambda$ such that $N \lea M$, let  $\gS_*(M)$ be the set of non-algebraic extensions of $\Gamma$ to $\gS^{na}(M)$.  Then it follows from  Fact \ref{inev}\footnote{The conditions here are slightly different as those of Fact \ref{inev}, as we only look at types over $\lea$-extensions of $N$ instead of types over all models in $\K_\lambda$. However  one can check that the proof of Theorem \ref{inev} goes through in this setting. } that for some $N\lea N^*\in \K_\lambda$, there is no $\gS_*$-inevitable set of types of size $\leq \lambda^+$; as otherwise $| \gS_*(N) | = |\Gamma| \leq \lambda^+$.

We build  $\langle N_i:i<\lambda^{++}\rangle$ by induction. Let $N_0$ be such that $N^*\lea N_0$, $\|N_0|-|N^*\|=\lambda$, and there is $a_0\in |N_0|-|N_i|$ realizing a type from $\gS_*(N^*)$. Let a realization of this type be $a_0$. At stage $i$, choose $N_i$ such that $N_i$ omits $\bigcup_{j<i}\{\type(c/N^*,N_j): c\in |N_j|-|N_*|\}$, $\|N_i|-|N^*\|=\lambda$, and $N_i$ realizes $p$ a type from $\gS_*(N^*)$. We can find such $N_i$ since the set is not $\gS_*$-inevitable. Let a realization of $p$ be $a_i$. 
\end{proof}

\end{document}